    \documentclass[11pt]{amsart} 

\usepackage{amsmath,amssymb}
\usepackage{amsfonts}
\usepackage{amsthm}
\usepackage{latexsym}
\usepackage{graphicx}
\usepackage{epsfig}

\newtheorem{theorem}{Theorem}[section]
\newtheorem{lemma}[theorem]{Lemma}
\newtheorem{proposition}[theorem]{Proposition}

\newtheorem{corollary}[theorem]{Corollary}

\newtheorem{prop}{Proposition}[section]
\newtheorem{remark}[prop]{Remark}

\newenvironment{definition}[1][Definition]
           {\medbreak\noindent {\em #1. \enspace}}
           {\par \medbreak}

\makeatletter \@addtoreset{equation}{section} \makeatother

\def\ddt{\frac{d}{dt}}

\def\RR{{\mathrm R}}

\def\Rc{{\mathrm {Rc}}}

\def\SS{{\mathrm S}}

\begin{document}

\title{Harnack Estimates for Ricci Flow on a Warped Product}

\author{Hung Tran$^*$
}

\address{Department of Mathematics,
 Cornell University, Ithaca, NY 14853-4201}
\email{hungtran@math.cornell.edu}



\renewcommand{\subjclassname}{%
  \textup{2000} Mathematics Subject Classification}
\subjclass[2000]{Primary 53C44}

\date{\today}

\begin{abstract} In this paper, we study the Ricci flow on closed manifolds equipped with warped product metric $(N\times F,g_{N}+f^2 g_{F})$ with $(F,g_{F})$ Ricci flat. Using the framework of monotone formulas, we derive several estimates for the adapted heat conjugate fundamental solution which include an analog of G. Perelman's differential Harnack inequality in \cite{perelman1}.
\end{abstract}
\maketitle

\markboth{Hung Tran} {Harnack Estimates for Ricci Flow on a Warped Product}
\section{\textbf{Introduction}}
The Ricci flow, first introduced by R. Hamilton in \cite{H3}, has been studied extensively, particularly the last decade thanks to the seminal contribution by G. Perelman (\cite{perelman1}). In general, the parabolic nonlinear system possesses certain level of difficulty and it is plausible that more precise results can be obtained for any specific manifold. \\

In this paper, given $(F,g_{F})$ Ricci flat, we investigate  a closed manifold $M^{n+p}$ with warped product symmetry, namely $M^{n+p}=N^{n}\times F^{p}$ equipped with the warped product metric \footnote[1]{ The warped structure can also be defined without the Ricci flat assumption as in the classical work of T. Colding and J. Cheeger \cite{cheeger96} but that one is not preserved along the Ricci flow in general.}
\begin{equation}
\label{warp}
g_{M}=g_{N}+f^2 g_{F}=g_{N}+e^{2u}g_{F},
\end{equation} 
which evolves under the Ricci flow 
\begin{equation}
\frac{\partial}{\partial{t}}g_{M}=-2\text{Rc}_{M}.
\end{equation}
It  is clear that the metric is constant on each fiber and  the structure is preserved along the Ricci flow. 
Throughout this paper, we will use  $\text{Rm}_{X}$, $\text{Rc}_{X}$, $\RR_{X}$ to denote curvature tensor, Ricci curvature, and scalar curvature with respect to a Riemannian metric $g$ on manifold X. We'll also omit the subscript when the context is clear.\\

The case $m=3$ was studied first by X. Cao on isoperimetric estimates for $S^2\times S^{1}$  in \cite{isop}. His work preceded Perelman's papers. More recently, J. Lott and N. Sesum were able to prove definite classification results using the Gauss-Bonet theorem for surfaces \cite{LS11}. In higher dimension, such strong results are unexpected. \\

Our approach in the current paper is inspired by the framework of monotonicity formulae set up be Perelman and, thus, we are interested in higher dimensional case and look for improving general results using the symmetry of the warped structure.  

In particular, we derive several estimates for a fundamental solution to the adapted conjugate heat equation which include a Harnack inequality that is structurally similar but computationally different from Perelman's result described below. 
Harnack inequalities, which aim to compare the values at two different points of a solution to a partial differential equation, are a vast subject of research that can be traced back to the 19th century. On parabolic equations, the breakthrough result was obtained in \cite{LY86} where the authors established an estimate for any positive solution to the heat equation. Later, Hamilton \cite{hmatrix93} proved a matrix version of Li-Yau type estimate under slightly different assumptions and, furthermore, he brought this idea into the study of general geometric flows by proving an analogue for the Ricci flow in \cite{H93harnack}. In \cite{perelman1}, Perelman marked his contribution by establishing a Li-Yau-Hamilton (LYH) type inequality for fundamental solutions of the conjugate heat equation. \\

To describe Perelman's result, let $(M^{m}, \overline{g}(t))$, $0\leq t\leq T$, be a solution to the Ricci flow on a closed manifold M and $\overline{H}=(4\pi\tau)^{-m/2}e^{-\overline{h}}$ a fundamental solution of the conjugate heat equation: 
\begin{equation}
\Box^{\ast}\overline{H}=(-\partial_{t}-\triangle+\RR)\overline{H}=0, 
\end{equation}
centered at $(y,T)$. Define
\begin{equation}
\overline{v}=\Big((T-t)(2\triangle{\overline{h}}-|\nabla{\overline{h}}|^2+\RR)+\overline{h}-n\Big)\overline{H},
\end{equation}
then 
\begin{equation}
\overline{v} \leq 0 \text{ for all $t<T$.}
\end{equation}
The result is important in the study of Ricci flow because these fundamental solutions are essential to understand monotone functionals which are the main novel machinery developed by Perelman. For instance, the recent work of X. Cao and Q. Zhang used this Harnack inequality to derive several estimates for the heat kernel and they eventually obtained rigidity for singularity model under type I assumption \cite{cz11}.\\

For warped product metrics, a fundamental solution $\overline{H}$ is not constant on each fiber since, for $x,y$ coordinates in N and $z$ in $F$, 
\begin{equation*}
\lim_{t\rightarrow T}\overline{H}((x,z_{1}),t;(y,z_{2}),T)=\delta_{(y,z_{2})}((x,z_{1})).
\end{equation*}
Thus the fundamental solution is not pertinently compatible with the metric setting and, therefore, it is interesting to derive a more suitable estimate taking advantage of that symmetrical structure. From now on any function is constant on each fiber and any operator acts w.r.t N if not specified otherwise. Also for $S=\RR_{N}-p|\nabla{u}|^2$,
\begin{equation}
 \Box_{w}^{\ast}=-\partial_{t}-\triangle_{N}+S
\end{equation}
is the adapted conjugate heat operator. The main result of our paper is the following Harnack inequality.

\begin{theorem}
\label{harnackwarpedlaplacian}
Let $(M,g_{M}(t))$, $0\leq t\leq T$, be a solution to the Ricci flow and $g_{M}(0)$ is a warped product metric as in (\ref{warp}). Let $\overline{H}$ be a positive function on M such that,  $H=\overline{H}e^{u}=(4\pi\tau)^{-n/2}e^{-h}$ is the fundamental solution of  
\begin{equation}
\label{conjheat}
\Box_{w}^{\ast}H+p\nabla{u}\nabla{H}=0 
\end{equation}
on $N$,  centered at $(y,T)$. Let $$v=\Big((T-t)(2\triangle{h}-|\nabla{h}|^2+S)+h-n\Big)H,$$ then  for all $t<T$, $v\leq 0.$
\end{theorem}    

An immediate consequence is the following LYH-type Harnack estimate.
\begin{corollary}
\label{harnackwarpedtime}
Under the assumption as above, along any smooth curve $\gamma(t)$ in N, let $\tau=T-t$, we have
\begin{align*}
-\partial_{t}h(\gamma(t),t)&\leq \frac{1}{2}(S(\gamma(t),t)+|\dot{\gamma}(t)|^2)-\frac{1}{2(T-t)}h(\gamma(t),t),\\
\partial_{\tau}(2\sqrt{\tau}h)&\leq \sqrt{\tau}(S+|\dot{\gamma}(t)|^2). 
\end{align*}
\end{corollary}
The above inequality is the motivation for the definition below. 

\begin{definition}
Given $\tau(t)=T-t$, we define the $\mathcal{L}_{w}$-length of a curve $\gamma: [\tau_{0},\tau_{1}]\mapsto N$, $[\tau_{0},\tau_{1}]\subset[0,T]$ by $\mathcal{L}_{w}(\gamma):=\int_{\tau_{0}}^{\tau_{1}}\sqrt{\tau}(S(\gamma(\tau))+|\dot{\gamma}(\tau)|^2)d\tau$. \\
For a fixed point $y\in N$ and $\tau_{0}=0$, the backward reduced distance is defined as 
\begin{equation}
\label{adaptedRD}
\ell_{w}(x,\tau_{1}):=\inf_{\gamma\in \Gamma}\{\frac{1}{2\tau_{1}}\mathcal{L}_{w}(\gamma)\}
\end{equation}
where $\Gamma=\{\gamma:[0,\tau_{1}]\mapsto M, \gamma(0)=y, \gamma(\tau_{1})=x\}$.\\
The backward reduced volume is 
\begin{equation}
V_{w}(\tau):=\int_{M}(4\pi\tau)^{-n/2}e^{-\ell_{w}(y,\tau)}d\mu_{\tau}(y).
\end{equation}
\end{definition}
The next theorem exposes relations between fundamental solutions and the reduced distance defined with respect to the same reference point $(y,T)$.
 
\begin{theorem} \label{conjandRD}
Let $(M,g_{M}(t))$, $0\leq t\leq T$, be a solution to the Ricci flow and $g_{M}(0)$ is a warped product metric as in (\ref{warp}). Let $\overline{H}$ be a positive function on M such that,  $H=\overline{H}e^{u}=(4\pi\tau)^{-n/2}e^{-h}$ is the fundamental solution of  $\Box_{w}^{\ast}H+p\nabla{u}\nabla{H}=0$ on $N$,  centered at $(y,T)$. If $\Phi$ be a  positive solution to the heat equation $\partial_{t}\Phi=\triangle_{\overline{g}}\Phi$, then the followings hold:
\begin{align*}
\text{\bf a. } &  h(x,l;y,t) \leq \ell_{w}(x,T-l),\\
\text{\bf b. } &  \lim_{\tau\rightarrow 0} 4\tau\ell_{w}(x,\tau)= d^{2}_{T}(y,x), \\
\text{\bf c. } & \lim_{\tau\rightarrow 0} \int_{N} hH\Phi d\mu_{N} =\lim_{\tau\rightarrow 0} \int_{N} \ell_{w}(x,\tau) H\Phi d\mu_{N}\\
&=\lim_{\tau\rightarrow 0} \int_{N} \frac{d^{2}_{T}(x,y)}{4\tau} H\Phi d\mu_{N}=\frac{n}{2}\Phi(y,T).\nonumber
\end{align*}
\end{theorem}
\begin{remark} If $H$ satisfies (\ref{conjheat}), then $\overline{H}=He^{-u}$ satisfies the conjugate heat equation on $(M,\overline{g})$. However, $\overline{H}$ is not a fundamental solution because it blows up on the whole fiber over $(y,T)$. That partially explains the following result which is interesting because if $\tilde{H}$ \textbf{was} a fundamental solution then the limit must be zero. 
\end{remark}
\begin{corollary} 
\label{strangebehavior} Let $\Psi$ is the general entropy as given in equation (\ref{generalentropy}) and $\overline{H}$ as above. If $\widetilde{H}=\frac{1}{V(F)}\overline{H}=(4\pi\tau)^{-(n+p)/2}e^{-\widetilde{h}}$ for $V(F)$ denotes the volume of $(F,g_{F})$ then $\lim_{\tau\rightarrow 0}\Psi(g_{M},\tau,\widetilde{h})=\infty$.
\end{corollary}
The organization of this paper is as follows. In section 2, we discuss the adaptation of the Ricci flow for a warped product and an equivalent system obtained via diffeomorphisms. In section 3, we derive modified monotonicity formulas and introduce  associated monotone functionals. In section 4, we prove several gradient estimates w.r.t the equivalent system. Section 5 collects the proofs of above results.\\

{\bf Acknowledgements:} The author is grateful to Professor Xiaodong Cao for continuous discussion and encouragement. 

\section{Basic Setting for Ricci Flow on Warped Products}

In this section, we will give a brief review of basic equations for the Ricci flow on a warped product and discuss diffeomorphisms that transform the system. 

\subsection{Basics of Ricci Flow on Warped Products}

Let $(M,\overline{g})$ be a warped product as in (\ref{warp}) then by standard computation (for example, see \cite{besse}), 
\begin{align}
\triangle_{g_{M}}h&=\triangle_{g_{N}}h+p\left\langle{\nabla{u},\nabla{h}}\right\rangle_{g_{N}},\\
d\mu_{M}&=d\mu_{N} f^{p} d\mu_{F},\\
\label{Rconwarp1}
\text{Rc}_{M} &= \text{Rc}_{N}-\frac{p}{f}\text{Hess}(f)-(f\triangle f+(p-1)|\nabla{f}|^2)g_{F},\\
\RR_{M} &= \RR_{N}-2p\frac{\triangle f}{f}-p(p-1)\frac{|\nabla{f}|^2}{f^2}.
\end{align}
\begin{remark} The Laplacian and Ricci curvature formulas suggest some connection to the analysis on Bakry-Emery Ricci tensor. 
\end{remark}
\begin{lemma}
Let $(M,g_{M}(t))$, $0\leq t\leq T$, be a solution to the Ricci flow and $g_{M}(0)$ is a warped product metric as in (\ref{warp}). Ricci flow preserves that warped structure and, in terms of the components, it is given by the following system:
\begin{align}
\label{warpedbeforelie}
\frac{\partial (g_{N})_{ij}}{\partial t} &=-2(\text{Rc}_{N})_{ij}+2p\frac{f_{ij}}{f}\\
\frac{df}{dt} &=\triangle_{g} f +(p-1)\frac{|\nabla{f}|^2}{f} \nonumber \\
\frac{du}{dt} &=\triangle_{g_{N}} u+p|\nabla u|^2=\Delta_{g_{M}}u \nonumber
\end{align}
\end{lemma}
\begin{proof}
Suppose $(g_{N},f)$ evolves as above then we can check that $g_{M}$ evolves by the Ricci flow. By the uniqueness theorem for Ricci flow, the result follows.
\end{proof}
Since u satisfies the heat equation, the maximum principle applies that if $u(.,0)\leq C$ then $u(.,t)\leq C$ as long as the flow exists. 
Furthermore, extensive use of the maximum principle yields interior estimates.
\begin{lemma}
\label{generalbound}
Let $(M,g_{M}(t))$, $0\leq t\leq T$, be a solution to the Ricci flow and $g_{M}(0)$ is a warped product metric as in (\ref{warp}). Then for each $\alpha>0$, there exists a constant $C(m,n,\alpha)$ such that if
$$|Rm|_{M}(.,t)<k \text{ for all $t\in [0,\frac{\alpha}{k}]$} $$
then 
$$ |\nabla^{m}u|_{\overline{g}(t)}\leq \frac{C|u(.,0)|_{{L^{\infty}}}}{t^{m/2}} $$
for all $t\in [0,\frac{\alpha}{k}].$
\end{lemma}
\begin{proof}
Since $\frac{\partial}{\partial t}u=\triangle_{g_{M}}u$ and $|u(.,0)|_{L^{\infty}}$ is preserved, the method of Shi's estimates applies. For a detailed calculation, see lemma 3.6 of \cite{brendlebook10}.  
\end{proof}
\begin{remark} The essence of this lemma is that the constant only depends on degree and dimension. Therefore, under suitable dilation limit analysis, it holds for any small compact interval under a uniform curvature bound.
\end{remark}

\subsection{Transform by Diffeomorphisms}
Now we consider the family of diffeomorphisms generated by $-p\nabla u$, $\frac{\partial}{\partial t}\varphi (t)(x)=(-p\nabla u)(\varphi(t)(x))$. Pullbacks $\tilde{g}(t)=\varphi^{\ast}(t)(g(t))$, $\tilde{u}(t)=\sqrt{p}\varphi^{\ast}(t)(u(t))=\sqrt{p} u(t)\circ \varphi(t)$ yield
\begin{align*}
\frac{\partial}{\partial t}\tilde{g}(t)& =L_{-p\nabla u}(\varphi^{\ast}(t)(g(t)))+\varphi^{\ast}(t)(\frac{\partial}{\partial t}g(t))\\
&=\varphi^{\ast}(t)(\frac{\partial}{\partial t}g(t)+L_{-p\nabla u}g(t))=\varphi^{\ast}(t)(-2\text{Rc}+2p du\otimes du)\\
&=-2\widetilde{\Rc}+2 d\tilde{u}\otimes \tilde{u},\\
\frac{\partial}{\partial t}\tilde{u}(t)& =\sqrt{p}L_{-p\nabla u}(\varphi^{\ast}(t)(u(t)))+\sqrt{p}\varphi^{\ast}(t)(\frac{\partial}{\partial t}u(t))\\
&=\sqrt{p}\varphi^{\ast}(t)(\frac{\partial}{\partial t}u(t)+L_{-p\nabla u}u(t))=\sqrt{p}\varphi^{\ast}(t)(\triangle u)=\triangle \tilde{u}.
\end{align*}
So system (\ref{warpedbeforelie}) is transformed into the following system on N (we abuse notation here as tildes are removed but the context should make clear what system is used)
\begin{align}
\mathcal{S}&=\Rc-du\otimes du \nonumber\\
\label{warpedLie}
\frac{\partial g}{\partial t} &=-2\text{Rc}+2du\otimes du=-2\mathcal{S}\\
\frac{\partial u}{\partial t} &=\triangle u.\nonumber
\end{align}
\begin{remark} The observation above implies that results from \cite{LS11} extend to a slightly more general setting: the fiber can be any Ricci flat manifold instead of $S^1$.
\end{remark}
\noindent Then the Christoffel symbols evolve by
\begin{align*}
\frac{\partial}{\partial t}\Gamma_{ij}^{k}&=-g^{kl}(\nabla_{i}\mathcal{S}_{jl}+\nabla_{j}\mathcal{S}_{il}-\nabla_{l}\mathcal{S}_{ij})\\
&=g^{kl}(-\nabla_{i}\text{Rc}_{jl}-\nabla_{j}\text{Rc}_{il}+\nabla_{l}\text{Rc}_{ij}+2\nabla_{i}\nabla_{j}u \partial_{l}u).
\end{align*}
\begin{lemma} 
\label{evolLap}
If $(N,u(.,t),g(t))$ is a solution to (\ref{warpedLie}) then the Laplacian acting on function evolves by 
\begin{align}
\frac{\partial}{\partial t}\Delta=2\mathcal{S}_{ij}\cdot \nabla_{i}\nabla_{j}-2\Delta{u}\left\langle{\nabla{u},\nabla(.)}\right\rangle
\end{align}\end{lemma}
\begin{proof} We compute
\begin{align*} 
\frac{\partial}{\partial t}\Delta&=\frac{\partial}{\partial t}(g^{ij}\nabla_{i}\nabla_{j})=\frac{\partial}{\partial t}(g^{ij}(\partial_{i}\partial_{j}-\Gamma_{ij}^{k}\partial_{k}))\\
&=(-\frac{\partial}{\partial t}g_{ij})\nabla_{i}\nabla_{j}-g^{ij}(\frac{\partial}{\partial t}\Gamma_{ij}^{k})\partial_{k}.
\end{align*}
Using the evolution equation for $\Gamma_{ij}^{k}$ yields
\begin{align*}
g^{ij}(\frac{\partial}{\partial t}\Gamma_{ij}^{k})\partial_{k}&=g^{kl}(-2g^{ij}\nabla_{i}\text{Rc}_{jl}+\nabla_{l}\RR)+2g^{kl}\Delta{u}\partial_{l}u\partial_{k}&\\
&=2\Delta{u}\left\langle{\nabla{u},\nabla(.)}\right\rangle,
\end{align*}
where we use the contracted 2nd Bianchi identity. The result follows immediately. 
\end{proof}

Now we derive evolution equations for some geometrical quantities. Recall $S=\text{tr}(\mathcal{S})=\RR-|\nabla u|^2$ and we compute:
\begin{align*}
\frac{\partial}{\partial t}|\nabla u|^2 &=\frac{\partial}{\partial t}(g^{ij}\nabla_{i}u\nabla_{j}u)=2\mathcal{S}(\nabla u,\nabla u)+2\left\langle{\nabla u,\nabla \frac{\partial}{\partial t}u}\right\rangle,\\
&=2\text{Rc}(\nabla u,\nabla u)-2|\nabla {u}|^4+2\left\langle{\nabla u,\nabla \triangle u}\right\rangle,\\
\triangle |\nabla u|^2 &=2\left\langle{\nabla u,\nabla \frac{\partial}{\partial t}u}\right\rangle+2\text{Rc}(\nabla u,\nabla u)+2|\text{Hess}u|^2 \text{ (Bochner's formula).}
\end{align*}
Combining equations above yields
\begin{equation}
\label{evolgradu}
\Box |\nabla u|^2 =-2|\text{Hess}u|^2-2|\nabla {u}|^4.
\end{equation}
\noindent If $\frac{\partial}{\partial t}g=v$ then $\frac{\partial}{\partial t}R=-\triangle{\text{trace}(v)}+\text{div}(\text{div}v)-(v,\text{Rc})$. In our case
\begin{align*}
\text{div}(\text{div}2\text{Rc})&=\nabla_{i}2\nabla_{j}\text{Rc}_{ij}=\nabla_{i}\nabla_{i}\RR=\triangle \RR,\\
\text{div}(\text{div} du\otimes du) &=\nabla_{i}(\nabla_{j}(\nabla_{i}u\nabla_{j}u))=\frac{1}{2}\triangle |\nabla u|^2+\left\langle{\nabla u,\nabla \triangle u}\right\rangle+|\triangle u|^2,\\
\frac{\partial}{\partial t} \RR &=-\triangle(-2S)-\triangle\RR+\triangle |\nabla u|^2+2\left\langle{\nabla u,\nabla \triangle u}\right\rangle\\
&+2|\triangle u|^2+2|\text{Rc}|^2-2\text{Rc}(\nabla u,\nabla u)\\
&=\triangle S+2\left\langle{\nabla u,\nabla \triangle u}\right\rangle+2|\triangle u|^2+2|\text{Rc}|^2-2\text{Rc}(\nabla u,\nabla u).
\end{align*} 
Combining equations above yields
\begin{equation}
\label{evolS}
\frac{\partial}{\partial t}S =\triangle S+2|\triangle u|^2+2|\mathcal{S}_{ij}|^2.
\end{equation}
\begin{remark} System (\ref{warpedLie}) and some evolution equations above appeared in \cite{List08} with a constant $\alpha_{n}$ associated with the term $du\otimes du$. However, in case $\alpha_{n}\geq 0$ if letting $\tilde{u}=\sqrt{\alpha_{n}}u$ recovers (\ref{warpedLie}). So every result in section 4 holds for $\alpha_{n}\geq 0$ as well.    
\end{remark}
\begin{remark} A generalization of that system is so-called the Ricci-Harmonic flow first introduced by R. Muller in \cite{Muller09} and it is interesting to extend the result here for that setting.
\end{remark}

\section{\textbf{Monotinicity Formulae}}
We shall derive the adapted and modified forms of monotonicity formulas and associated functionals, first introduced by Perelman in \cite{perelman1}, to the warped product setting in (\ref{warpedbeforelie}) and (\ref{warpedLie}).\\

For a Ricci flow solution on a closed manifold $(M^{m},g_{M}(t))$, Perelman introduced the following funtionals.
\begin{equation}
\text{Energy: }\mathcal{F}(g_{M},\overline{h})=\int_{M}(\RR_{M}+|\nabla\overline{h}|^2)e^{-\overline{h}}d\mu_{M}
\end{equation}
restricted to $\int_{M}e^{-\overline{h}}d\mu_{M}=1$.
\begin{equation}
\label{generalentropy}
\text{Entropy: } \Psi(g_{M},\tau,\overline{h})=\int_{M}\Big(\tau(|\nabla \overline{h}|^2+\RR_{M})+\overline{h}-m\Big)(4\pi\tau)^{-m/2}e^{-\overline{h}}d\mu_{M}
\end{equation}
restricted to $\int_{M}(4\pi\tau)^{-m/2}e^{-\overline{h}}d\mu_{M}=1$.\\
Furthermore, he computed the evolution equations for these functionals when the test function $\overline{H}$($=e^{-\overline{h}}$ for energy or $=(4\pi\tau)^{-m/2}e^{-\overline{h}}$ for entropy) satisfies the conjugate heat equation $(-\partial_{t}-\triangle+\RR_{M})\overline{H}=0$ and obtained monotone formulae:
\begin{align}
\frac{d}{dt}\mathcal{F}&=2\int_{M}\Big(|\text{Rc}_{M}+\text{Hess}_{M}(\overline{h})|^2\Big)e^{-\overline{h}}d\mu_{M},\\
\frac{d}{dt}\Psi&=2\tau\int_{M}\Big(|\text{Rc}+\text{Hess}_{M}\overline{h}-\frac{g_{M}}{2\tau}|^2\Big)(4\pi\tau)^{-m/2}e^{-\overline{h}}d\mu_{M}.
\end{align}
First, to adapt these formulas to our setting, we observe the following relations.
\begin{lemma}
\label{baseandwhole}
{\bf a.} If $\overline{H}=e^{-\overline{h}}$ and $H=\overline{H}e^{pu}=e^{-h}$ then 
\begin{align*}
h &=\overline{h}-pu, \\
\overline{h}_{t}&=|\nabla\overline{h}|^2-\triangle_{M}\overline{h}-\RR_{M} \text{ iff}\\
h_{t}&=-S-\triangle h+\nabla h(\nabla h+p\nabla u).
\end{align*}
{\bf b.} If $\overline{H}=(4\pi\tau)^{-(n+p)/2}e^{-\overline{h}}$ and $H=\overline{H}e^{pu}=(4\pi\tau)^{-n/2}e^{-h}$ then 
\begin{align*}
h &=\overline{h}-pu+\frac{p}{2}\ln(4\pi\tau) \text{ and}\\
\overline{h}_{t}&=|\nabla\overline{h}|^2-\triangle_{M}\overline{h}-\RR_{M}+\frac{n+p}{2\tau} \text{ iff}\\
h_{t}&=-S-\triangle h+\nabla h(\nabla h+p\nabla u)+\frac{n}{2\tau}.
\end{align*}
\end{lemma}

\begin{proof} 

{\bf a.} We compute 
\begin{align*}
\overline{h}_{t}&=|\nabla\overline{h}|^2-\triangle_{M}\overline{h}-\RR_{M}\\
&=|\nabla\overline{h}|^2-\triangle\overline{h}-p\nabla\overline{h}\nabla{u}-\RR_{N}+2p\triangle u+p(p+1)|\nabla{u}|^2,\\
h_{t}&=\overline{h}_{t}-pu_{t}\\
&=|\nabla\overline{h}|^2-\triangle_{g}\overline{h}-p\nabla\overline{h}\nabla{u}-\RR_{N}+2p\triangle u+p(p+1)|\nabla u|^2\\
&-p\triangle u-p^2|\nabla{u}|^2\\
&=-\triangle h-\RR_{N}+p|\nabla u|^2+\nabla (h+pu)\nabla h.
\end{align*}
{\bf b.} This follows from a similar computation. 
\end{proof}

\begin{lemma}
\label{monoformwarp}
Adapted to the Ricci flow on warped product metric given in (\ref{warp}) , the monotonicity formulas are given by\\
{\bf a.} Energy: $ \mathcal{F}(g,u,h)=\int_{N}(S+|\nabla h|^2)e^{-h}d\mu_{N}$ restricted to $\int_{N}e^{-h}d\mu_{N}=\frac{1}{V(F)}$.\\
Furthermore if  $h_{t}=-S-\triangle h+\nabla h(\nabla h+p\nabla u)$ then
\begin{align*}
\frac{d}{dt}\mathcal{F}&=2\int_{N}\Big(|\mathcal{S}+\text{Hess}(h)|^2+p|\triangle u-\nabla{u}\nabla{h}|^2\Big)e^{-h}d\mu_{N}.
\end{align*}
{\bf a'.} W.r.t system (\ref{warpedLie}), $h_{t}=-S-\triangle h+|\nabla h|^2$.\\
{\bf b.} Entropy:  Restricted to $\int_{N}H d\mu_{N}=\int_{N}(4\pi\tau)^{-n/2}e^{-h}d\mu_{N}=\frac{1}{V(F)}$,\\
\[
\Psi(g,u,\tau,h)=\int_{N}\Big[\tau(|\nabla h|^2+S)+(h+pu-n-p)-\frac{p}{2}\ln(4\pi\tau)\Big]Hd\mu_{N}.
\]
And if $h_{t}=-S-\triangle h+\nabla h(\nabla h+p\nabla u)+\frac{n}{2\tau}$  then
\begin{align*}
\frac{d}{dt}\Psi&=2\tau\int_{N}(|\mathcal{S}+\text{Hess}{h}-\frac{g}{2\tau}|^2+p|\triangle u-\nabla{u}\nabla{h}+\frac{1}{2\tau}|^2)(4\pi\tau)^{-n/2}e^{-h}d\mu_{N}.
\end{align*}
{\bf b'.} W.r.t system (\ref{warpedLie}), $h_{t}=-S-\triangle h+|\nabla h|^2+\frac{n}{2\tau}.$
\end{lemma}

\begin{proof}
{\bf a.} The entropy is given by
\begin{align*}
\mathcal{F}(\overline{g},\overline{h})&=\int_{M}(\RR_{M}+|\nabla\overline{h}|^2)e^{-\overline{h}}d\mu_{M}\\
&=\int_{N}\int_{F}(\RR_{N}-2p\triangle u-p(p+1)|\nabla u|^2+|\nabla\overline{h}|^2)e^{-\overline{h}}e^{pu}d\mu_{N}d\mu_{F}\\
&=V(F)\int_{N}(\RR-p|\nabla u|^2+|\nabla h|^2)e^{-h}d\mu_{N},
\end{align*}
where we use integration by parts (IBP) to simplify
\[
\int_{N}2p\Delta{u}e^{-h}d\mu_{N}=\int_{N}2p\nabla{h}\nabla{u}e^{-h}d\mu_{N}.
\]
Furthermore, if 
\begin{align*}
\overline{h}_{t}&=|\nabla\overline{h}|^2-\triangle_{\overline{g}}\overline{h}-\RR_{M}, \text{ then}\\
\frac{d}{dt}\mathcal{F}&=2\int_{M}(|\text{Rc}_{M}+\text{Hess}_{g_{M}}\overline{h}|^2d\mu_{M}\\
&=2V(F)\int_{N}\Big(\Big|\text{Rc}-p du\otimes du-p \text{Hess}(u)+\text{Hess}(h+pu)\Big|^2\\
&+p\Big|-\triangle{u}-p|\nabla{u}|^2+\nabla{u}\nabla(h+pu)\Big|^2\Big)d\mu_{N}\\
&=2V(F)\int_{N}(|\text{Rc}_{g}-pdu\otimes du+\text{Hess}(h)|^2+p|\triangle u-\nabla{u}\nabla{h}|^2)e^{-h}d\mu_{N}.
\end{align*}
The result then follows from lemma \ref{baseandwhole}.\\
{\bf a'.} It follows from $L_{-p\nabla u}h=-p\nabla u\nabla h$.\\
{\bf b.} and {\bf b'.} are similar using part b) of lemma \ref{baseandwhole}. 
\end{proof}

\begin{corollary}\label{entropyadaptedwarp}
If $ \Psi_{w}(g,u,\tau,h)=\int_{N}\Big(\tau(|\nabla h|^2+S)+(h-n)\Big)(4\pi\tau)^{-n/2}e^{-h}d\mu_{N}$ and $h_{t} =-S-\triangle h+\nabla h(\nabla h+p\nabla u)+\frac{n}{2\tau}$, then
\begin{equation*}
\ddt \Psi_{w}=2\tau\int_{N}(|\mathcal{S}+\text{Hess}h-\frac{g}{2\tau}|^2+p|\triangle u-\nabla{u}\nabla{h}|^2)(4\pi\tau)^{-n/2}e^{-h}d\mu_{N}.
\end{equation*}
\end{corollary}

\begin{proof}
We have 
\begin{equation*}
\Psi_{w}=\Psi+\int_{N}(p-u+\frac{p}{2}\ln(4\pi\tau))\overline{H}e^{pu} d\mu_{N}.
\end{equation*}
Since u satisfies the heat equation on M and $\overline{H}$ the conjugate, $\frac{d}{dt}\int_{N}u\overline{H}d\mu_{M}=0$.  Thus, 
\begin{align*}
\frac{d}{dt} \Psi_{w}&=\frac{d}{dt}\Psi+\frac{p}{2}\Big(\frac{d}{dt}\ln(4\pi\tau)\Big)\int_{N}\overline{H}e^{pu} d\mu_{N}\\
&=\frac{d}{dt}\Psi-\frac{p}{2\tau}\int_{N}\overline{H}e^{pu} d\mu_{N}.
\end{align*}
On the other hand,
\begin{align*}
|\triangle u-\nabla{u}\nabla{h}+\frac{1}{2\tau}|^2&=|\triangle u-\nabla{u}\nabla{h}|^2+\frac{1}{4\tau^2}+\frac{1}{\tau}(\triangle u-\nabla{u}\nabla{h}),\\
\int_{N}\triangle u e^{-h}d\mu_{N}&=\int_{N}\nabla{u}\nabla{h}e^{-h}d\mu_{N} \text{ by Stoke's theorem.}
\end{align*}
The result follows. 
\end{proof}
An immediate application from the above calculation is the following result.
\begin{proposition}
Let $(M,g_{M})$ be a warped product given as in (\ref{warp}) manifold N. If M is a gradient soliton and the soliton function is constant on each fiber then $(N,g_{N})$ is Ricci flat and f is a constant function.
\end{proposition}
\begin{proof}
Suppose $(M,g)$ is a gradient soliton satisfying \[\Rc_{M}+\text{Hess}\overline{h}=\lambda g_{M},\]
with $\overline{h}$ constant on each fiber. Let $h=\overline{h}+pu$ and follow the calculation from previous lemmas, we obtain:
\begin{align*}
0=\int_{M}|Rc_{M}+\text{Hess} \overline{h}-\lambda {g}_{M}|^2 e^{-\overline{h}} d\mu_{M}&=\int_{M}|\mathcal{S}+\text{Hess}(h)-\lambda{g}|^{2} e^{-\overline{h}}d\mu_{M}\\
&+V(F)\int_{N}p|\Delta u-\nabla{u}\nabla{h}+\lambda|^2 d\mu_{N}.
\end{align*}
On the other hand,
\begin{align*}
|\triangle u-\nabla{u}\nabla{h}+\lambda|^2&=|\triangle u-\nabla{u}\nabla{h}|^2+\lambda^2+2\lambda(\triangle u-\nabla{u}\nabla{h}),\\
\int_{N}\triangle u e^{-h}d\mu_{N}&=\int_{N}\nabla{u}\nabla{h}e^{-h}d\mu_{N} \text{ by Stoke's theorem.}
\end{align*}

Thus $\lambda=0$ and $(N\times F,g_{N}+f^2 g_{F})$ is a gradient steady soliton. As $N\times F$ is closed, by either theorem 2.4 of \cite{perelman1} or 20.1 of \cite{Hsurvey}, the manifold is Ricci flat. That is 
\begin{align*}
0&=f\triangle {f}+(p-1)|\nabla{f}|^2=\Delta{u}+p|\nabla{u}|^2,\\
0&=\text{Rc}(g)-p\frac{\text{Hess}_{g}(f)}{f}.
\end{align*}
However, as $\int_{N}\Delta{u}d\mu_{N}=0$, the first equality implies that $\nabla{u}=0$ and so f must be constant. Plugging into the 2nd equality yields the result.
\end{proof}
\begin{remark} Also computation above shows that monotone functionals in \cite{List08} are just suitable modification of ones developed by Perelman for warped products. For completeness, we'll repeat the definition here.
\end{remark}
\begin{definition}
Along the flow given by (\ref{warpedbeforelie}) or (\ref{warpedLie}), restricted to $\int_{N}e^{-h}d\mu_{N}=1$,  
\begin{equation} 
\mathcal{F}_{w}(g,u,h)=\int_{N}(S+|\nabla h|^2)e^{-h}d\mu_{N}.
\end{equation}
Restricted to $\int_{N}(4\pi\tau)^{-n/2}e^{-h}d\mu_{N}=1$,
\begin{equation}
\Psi_{w}(g,u,\tau,h)=\int_{N}\Big(\tau(|\nabla h|^2+S)+(h-n)\Big)(4\pi\tau)^{-n/2}e^{-h}d\mu_{N}.
\end{equation}
Furthermore, associated functionals can be defined similarly as follows:
\begin{align}
\mu_{w}(g,u,\tau)&=\inf_{h}{\Psi_{w}(g,u,h,\tau)},\\
\upsilon_{w}(g,u)&=\inf_{\tau>0}{\mu_{w}(g,u,\tau)},\\
\lambda_{w}(g,u)&=\inf_{h}{\mathcal{F}_{w}(g,u,h)}\geq \lambda(g_{M}) .
\end{align}
\end{definition}
\begin{remark} These functionals satisfy diffeomorphism and scaling invariance:
\begin{align*}
\Psi_{w}(g,u,\tau,h)&= \Psi_{w}(cg,u,c\tau,h),\\
\mu_{w}(g,u,\tau)&=\mu_{w}(cg,u,c\tau),\\
\upsilon_{w}(g,u)&=\upsilon_{w}(cg,u).
\end{align*}
\end{remark}

\begin{remark} The functionals are defined w.r.t (\ref{warpedLie}), which is more general than the Ricci flow system. Nevertheless, the parabolic structure is rigorous enough that these new quantities behave similarly. 
First, we collect some lemmas whose proof is identical to the counterpart for the Ricci flow. Curious readers should consult appropriate sections in \cite{chowetc1} and \cite{chowetc3}. 
\end{remark}
\begin{lemma} 
\label{basicmu}
For any metric g, smooth function u on closed N anf $\tau>0$, \\
{\bf a.} $\mu_{w}(g,u,\tau)$ is finite.\\
{\bf b.} $\lim_{\tau\rightarrow 0^{+}}\mu_{w}(g,u,\tau)=0$.\\
{\bf c.} There exists a smooth minimizer $f_{\tau}$ for $\Psi_{w}(g,u,.,\tau)$ which satisfies
\[
\tau(2\triangle{f_{\tau}}-|\nabla{f_{\tau}}|^2+\SS)+f_{\tau}-n=\mu_{w}(g,u,\tau).
\]
{\bf d.} Along the flow, $0\leq t_{1}\leq t_{2}\leq T$ and $\tau(t)>0$, $\frac{d}{dt}\tau=-1$ then
\begin{equation*} \mu_{w}(g(t_{2}),\tau(t_{2}))\geq \mu_{w}(g(t_{1}),\tau(t_{1})).
\end{equation*} 
\end{lemma}

\begin{lemma}
\label{behaviorforsmalltau} Assume as above then if $\lambda_{w}(g,u)>0$ then $\lim_{\tau\rightarrow \infty}\mu_{w}(g,u,\tau)=+\infty$.
\end{lemma}
Combining the last two lemmas yields
\begin{corollary}
\label{finitev}
If $\lambda_{w}(g,u)>0$ then $\upsilon(g,u)$ is well-defined and finite.
\end{corollary}


An immediate application of the monotone framework (particularly Corollary \ref{finitev}) is the theorem below which resembles a result of P. Topping in \cite{Topping05} using scalar curvature to control diameter for a compact manifold along the Ricci flow. The proof is omitted as it is identical without notable modification once the setting is up. 

\begin{theorem}
\label{simtop}
Let $n\geq 3$ and $(N^{n},g(t),u(.,t))$ be a solution to (\ref{warpedbeforelie}) with $\upsilon_{w}(g,u)\geq -\infty$ then there exists a C depending on n, $\upsilon_{w}(g,u)$ such that
$$\text{diam}(N,g)\leq C\int_{N}S_{+}^{(n-1)/2}d\mu_{N}=C\int_{N}(\RR_{N}-p|\nabla{u}|^2)_{+}^{(n-1)/2}d\mu_{N}.$$
\end{theorem}
\begin{remark} The $+$ subscript denotes the positive part and $C=\max\{\frac{12}{\omega_{n}},6e^{3^{n}37-\upsilon_{w}(g,u)}\}.$
\end{remark}
\begin{corollary}
Let $n\geq 3$ and Let $(M,g_{M}(t))$, $0\leq t\leq T$, be a solution to the Ricci flow and $g_{M}(0)$ is a warped product metric as in (\ref{warp}). Furthermore assume that $\lambda_{w}(g(0))>0$ then there exists $C_{1},C_{2}$ depending on the initial conditions such that
$$\text{diam}(M,g)\leq C_{1}+C_{2}\int_{N}(\RR_{N}-p|\nabla{u}|^2)_{+}^{(n-1)/2}d\mu_{N}.$$
\end{corollary}
\begin{proof}
Since the flow preserves the warped product setting, $(F,g_{F})$ is closed, $|u(.,t)|_{L^{\infty}}\leq |u(.,0)|_{L^{\infty}}$, the result follows from triangle inequalities and theorem \ref{simtop}. 
\end{proof}
\begin{remark}
Applying Topping result directly yields the bound $C\int_{N}(\RR_{N}-2p\triangle{u}-p(p+1)|\nabla{u}|^2)^{(n+p-1)/2}e^{pu}d\mu_{N}$. Thus, the above corollary gives a better estimate. 
\end{remark}

\section{\textbf{Gradient Estimates and Harnack Inequality}}
For this section, we restric ourselves to system (\ref{warpedLie}) and prove gradient estimates and a differential Harnack inequality for solutions to the conjugate heat equation. This section might be of independent interest  and some arguments here are similar to those in \cite{Ni06}. 

Recall $\Box_{w}^{\ast}=-\partial_{t}-\triangle+S$ and since $\frac{d}{dt}d\mu_{N}=S d\mu_{N}$ along the flow, $\Box_{w}^{\ast}$ is the adapted conjugate operator to the heat operator. Following standard theory on heat equations, for example  \cite[Chapter 23, 24]{chowetc3}, we denote $$H(x,t;y,T)=(4\pi(T-t))^{-n/2}e^{-h}=(4\pi\tau)^{-n/2}e^{-h},$$ for $\tau=T-t>0$, to be the heat kernel. That is, for fixed $(x,t)$, H is the fundamental solution of equation $\Box{H}=0$ based at $(x,t)$, and similarly for fixed $(y,T)$ and equation $\Box_{w}^{\ast}H=0$. The ultimate goal is to prove the following theorem.
\begin{theorem}
\label{eqstate}
Let $(N,u(.,t),g(t))$, $0\leq t\leq T$, be a solution to (\ref{warpedLie}). Fix $(y,T)$, let $H=(4\pi\tau)^{-n/2}e^{-h}$ be the  fundamental solution of $(-\partial_{t}-\triangle_{N}+S)H=0$,  and $$v=\Big((T-t)(2\triangle{h}-|\nabla{h}|^2+S)+h-n\Big)H,$$  then for all $t<T$, we have  $$v\leq 0. $$
\end{theorem}  
First let us recall the asymptotic behavior of the heat kernel as $t\rightarrow T$.
\begin{theorem}\cite[Theorem 24.21]{chowetc3} \label{asymptoticconj}
For $\tau=T-t$,
\begin{equation*} 
H(x,t;y,T) \sim \frac{e^{-\frac{d_{T}^{2}(x,y)}{4\tau}}}{(4\pi\tau)^{n/2}}\Sigma_{j=0}^{\infty}\tau^{j}u_{j}(x,y,\tau).
\end{equation*}
More precisely, there exist $t_{0}>0$ and a sequence $u_{j}\in C^{\infty}(M\times M\times [0,t_{0}])$ such that,
\begin{equation*}
H(x,t;y,T)-\frac{e^{-\frac{d_{T}^{2}(x,y)}{4\tau}}}{(4\pi\tau)^{n/2}}\Sigma_{j=0}^{k}\tau^{j}u_{j}(x,y,T-l)=w_{k}(x,y,\tau),
\end{equation*} 
with 
\begin{equation*}
u_{0}(x,x,0)=1,
\end{equation*}
and 
\begin{equation*}
w_{k}(x,y,\tau)=O(\tau^{k+1-\frac{n}{2}})
\end{equation*}
as $\tau\rightarrow 0$ uniformly for all $x,y\in M$.
\end{theorem}
 Next we derive a general estimate on the kernel. The proof is inspired by \cite{cz11}.
\begin{lemma}
\label{conjestimateS}
Let $B =-\inf_{0<\tau\leq T} \mu_{w}(g(0),\tau)$( B is well-defined as in lemma \ref{behaviorforsmalltau}) and $D =\min\{0,\inf_{N\times \{0\}}{S}\}$, then we have $$H(x,t,y,T) \leq e^{B-(T-t)D/3}(4\pi (T-t))^{-n/2}.$$
\end{lemma}

\begin{proof}
Without loss of generality, we may assume that $t=0$. Let $\Phi(y,t)$ be any positive solution to the heat equation along the flow. First, we obtain an upper bound for the $L^{\infty}$-norm of $\Phi(.,T)$ in terms of $L^{1}$-norm of $\Phi(.,0)$.\\
Set $p(l)=\frac{T}{T-l}=\frac{T}{\tau}$ then $p(0)=1$ and $\lim_{l\rightarrow T}p(l)=\infty$. For $A=\sqrt{\int_{N}\Phi^{p}d\mu}$, $v=A^{-1}\Phi^{p/2}$ and $ \nabla \Phi\nabla (v^2\Phi^{-1})=(p-1)p^{-2}4|\nabla v|^2$, integration by parts (IBP) yields 
\begin{align*}
\partial_{t}(\ln{||\Phi||_{L^{p}}})&=-p'p^{-2}\ln(\int_{N}\Phi^{p}d\mu)+(p\int_{N}\Phi^pd\mu)^{-1}\partial_{t}(\int_{N}\Phi^{p}d\mu)
\\&
=-p'p^{-2}\ln(\int_{N}\Phi^{p}d\mu)+(p\int_{N}\Phi^pd\mu)^{-1}\Big(\int_{N}\Phi^{p}(p\Phi^{-1}\Phi'+p'\ln{\Phi}-S)d\mu\Big)
\\&
=-p'p^{-2}\ln(A^2)+p^{-1}A^{-2}\Big(\int_{N}A^2v^2(p\Phi^{-1}\Phi'+p'\frac{2}{p}\ln{(Av)}-S)d\mu\Big)  
\\&
=\int_{N}v^2\Phi^{-1}\triangle \Phi d\mu+p'p^{-2}\int v^2\ln{v^2}-p^{-1}\int_{N} Sv^2d\mu
\\&
=p'p^{-2}\int_{N} v^2\ln{v^2}d\mu-(p-1)p^{-2}\int_{N}4|\nabla v|^2 d\mu-p^{-1}\int_{N} Sv^2d\mu
\\&
=p'p^{-2}\Big(\int_{N}v^2\ln{v^2}d\mu-\frac{p-1}{p'}\int_{N}4|\nabla v|^2d\mu-\frac{p-1}{p'}\int_{N}Sv^2d\nu\Big)
\\&
+((p-1)p^{-2}-p^{-1})\int_{N}Sv^2d\mu.
\end{align*} 
Note that if we set $v^2=(4\pi\tau)^{-n/2}e^{-h}$ then the first term becomes 
$$-p'p^{-2}\Psi_{w}(g,u,\frac{p-1}{p'},h)-n-\frac{n}{2}\ln({4\pi\frac{p-1}{p'}}).$$
We have
\[ p'p^{-2}=\frac{1}{T},  \frac{p-1}{p'}=\frac{l(T-l)}{T},\mbox{ and } (p-1)p^{-2}-p^{-1}=-\frac{(T-l)^2}{T^2}.\]
For $0<t_{0}<T$, $\tau(t_{0})=\frac{t_{0}(T-t_{0})}{T}$ and $\frac{d}{dt}{\tau}=-1$ then $0<\tau(0)=\frac{t_{0}(2T-t_{0})}{T}<T$.
By Lemma \ref{basicmu}, we arrive at
$$-p'p^{-2}\Psi_{w}(g(l),u,\frac{p-1}{p'},h)\leq -\frac{1}{T}\Psi_{w}(g(0),u,\tau(0),h)\leq -\frac{1}{T}\inf_{0<\tau\leq T} \mu_{w}(g(0),\tau)=\frac{B}{T}.$$
Thus $$T\partial_{t}(\ln{|\Phi||_{L^{p}}})\leq B-n-\frac{n}{2}\ln{(4\pi\frac{t(T-t)}{T})}-\frac{(T-t)^2}{T} D,$$ since, by (\ref{evolS}), the minimum of S is nondecreasing along the flow.
Integrating the above inequality yields
$$T\ln{\frac{||\Phi(.,T)||_{L^{\infty}}}{||\Phi(.,0)||_{L^{1}}}}\leq T(B-n-\frac{n}{2}(\ln{(4\pi T)}-2))-\frac{T^{2}}{3}D.$$
Then
$$||\Phi(.,T)||_{L^{\infty}}\leq e^{B-TD/3}(4\pi T)^{-n/2}||\Phi(.,0)||_{L^{1}}.$$
Since 
\begin{equation}
\label{heatbyheatkernel}
\Phi(y,T)=\int_{N}H(x,0,y,T)\Phi(x,0)d\mu_{g(0)}(x),
\end{equation}
and the above inequality holds for any arbitrary positive heat equation, we obtain 
$$H(x,0,y,T) \leq e^{B-TD/3}(4\pi T)^{-n/2}.$$
\end{proof}

\begin{lemma}
\label{gradconjS}Assume there exist $k_{1},k_{2},k_{3}\geq 0$ such that the followings hold on $N\times[0,T]$, 
\begin{align*}
\text{Rc}(g(t)) &\geq -k_{1}g(t),\\
\max\{S,|\nabla{S}|^2\} &\leq k_{2},\\
|\nabla{u}|^2 &\leq k_{3}.
\end{align*}
Let q be any positive solution to the equation $\Box_{w}^{\ast}q=0$ on $N\times [0,T]$ and $\tau=T-t$. If $q<A$ hen there exist $C_{1},C_{2}$ depending on $k_{1},k_{2},k_{3}$ and n such that for $0<\tau\leq \min\{1,T,\frac{1}{2k_{2}}\}$, we have
\begin{equation}
\tau \frac{|\nabla{q}|^2}{q^2}\leq (1+C_{1}\tau)(\ln{\frac{A}{q}}+C_{2}\tau).
\end{equation}

\end{lemma}
\begin{proof}
We compute
\begin{align*}
(-\partial_{t}-\triangle)\frac{|\nabla{q}|^2}{q}&=S\frac{|\nabla{q}|^2}{q}+\frac{1}{q}(-\partial_{t}-\triangle)|\nabla{q}|^2+2|\nabla{q}|^2\nabla{\frac{1}{q}}\nabla{\ln{q}}-2\nabla|\nabla{q}|^2\nabla{\frac{1}{q}},\\
\frac{1}{q}(-\partial_{t}-\triangle)|\nabla{q}|^2&=\frac{1}{q}\Big[-2\mathcal{S}(\nabla{q},\nabla{q})-2\text{Rc}(\nabla{q},\nabla{q})-2\nabla{q}\nabla{(Sq)}-2|\nabla^2{q}|^2 \Big],\\
2|\nabla{q}|^2\nabla{\frac{1}{q}}\nabla{\ln{q}}&=-2\frac{|\nabla{q}|^4}{q^3},\\
-2\nabla|\nabla{q}|^2\nabla{\frac{1}{q}}&=4\frac{\nabla^2{q}(\nabla{q},\nabla{q})}{q^2}.
\end{align*}
Thus 
\begin{align*}
(-\partial_{t}-\triangle)\frac{|\nabla{q}|^2}{q}&=\frac{-2}{q}|\nabla^2{q}-\frac{dq\otimes dq}{q}|^2+\frac{-4\text{Rc}(\nabla{q},\nabla{q})+2(\nabla{u}\nabla{q})^2-2\nabla{q}\nabla{(Sq)}}{q}+S\frac{|\nabla{q}|^2}{q}\\
&\leq [(4+n)k_{1}+3k_{3}+1]\frac{|\nabla{q}|^2}{q}+k_{2}q.
\end{align*}
Furthermore, we have
\begin{align*}
(-\partial_{t}-\triangle)(q\ln{\frac{A}{q}})&=-Sq\ln{\frac{A}{q}}+Sq+\frac{|\nabla{q}|^2}{q}\\
& \geq \frac{|\nabla{q}|^2}{q}-(nk_{1}+k_{3})q-k_{2}q\ln{\frac{A}{q}}.
\end{align*}
Let $\Phi=a(\tau)\frac{|\nabla{q}|^2}{q}-b(\tau)q\ln{\frac{A}{q}}-cq,$ and we can choose a,b,c appropriately such that $(-\partial_{t}-\triangle)\Phi\leq 0$. For example,
\begin{align*}
a&=\frac{\tau}{1+[(4+n)k_{1}+3k_{3}+1]\tau},\\
b&=e^{k_{2}\tau},\\
c&=(e^{k_{2}\tau}(nk_{1}+k_{3})+k_{2})\tau.
\end{align*}
Then by the maximum principle, noticing that $\Phi\leq 0$ at $\tau=0$, 
$$a\frac{|\nabla{q}|^2}{q}\leq b(\tau)q\ln{\frac{A}{q}}+cq.$$
The result then follows from simple algebra.\\
\end{proof}
The next result, mainly from \cite{Muller10}, relates the reduced distance defined in (\ref{adaptedRD}) with the distance at time T.
\begin{lemma}
\label{compareRDandconjheat}
Let $L_{w}(x,\tau)=4\tau\ell_{w}(x,\tau)$ then we have.\\
{\bf a.} Assume that there exists $k_{1},k_{2}\geq 0$ such that $-k_{1}g(t)\leq \mathcal{S}(t) \leq k_{2}g(t)$ for $t\in[0,T]$ then $L_{w}$ is smooth amost everywhere and a local Lipschitz function on $N\times [0,T]$. Furthermore, 
$$e^{-2k_{1}\tau}d_{T}^2(x,y)-\frac{4k_{1}n}{3}\tau^2\leq L_{w}(x,\tau)\leq e^{2k_{2}\tau}d_{T}^2(x,y)+\frac{4k_{2}n}{3}\tau^2 .$$
{\bf b.} $\Box_{w}^{\ast}\Big(\frac{e^{-\frac{L_{w}(x,\tau)}{4\tau}}}{(4\pi\tau)^{n/2}}\Big)\leq 0.$\\
{\bf c.} $H(x,t;y,T)=(4\pi\tau)^{-n/2}e^{-h}$ then $h(x,t;y,T)\leq \ell_{w}(x,T-t)$ . 
\end{lemma}
\begin{proof}
{\bf a.} This follows from the result \cite[Lemma 4.1]{Muller10} for general flows.\\
{\bf b.} This follows from \cite[Lemma 5.15]{Muller10}. The key assumption is the non-negativity of the quantity,
\begin{equation*}
\mathcal{D}(\mathcal{S},X)=\partial_{t}S-\triangle{S}-2|\mathcal{S}|^2+4(\nabla_{i}\mathcal{S}_{ij})X_{j}-2(\nabla_{j}S)X_{j}+2(\text{Rc}-\mathcal{S})(X,X).
\end{equation*} 
In our case, applying (\ref{evolS}) and the second Bianchi identity yields
\begin{align*}
\mathcal{D}(\mathcal{S},X)&=2(\triangle{u})^2+4\nabla^{i}(\RR_{ij}-u_{i}u_{j})X^{j}-2\nabla_{j}(\RR-|\nabla{u}|^2)X^{j}+2du\otimes du(X,X)\\
&=2(\triangle{u})^2-4\triangle{u}\left\langle{\nabla{u},X}\right\rangle+2\left\langle{\nabla{u},X}\right\rangle^2=2(\triangle{u}-\left\langle{\nabla{u},X}\right\rangle)^2\geq 0.
\end{align*}
{\bf c.} We first observe that part a) implies $\lim_{\tau\rightarrow 0}L_{w}(x,\tau) =d_{T}^2(y,x)$ and, hence,
\[\lim_{\tau\rightarrow 0}\frac{e^{-\frac{L_{w}(x,\tau)}{4\tau}}}{(4\pi\tau)^{n/2}}=\delta_{y}(x),\]
since locally Riemannian manifolds look like Euclidean. It follows immediately from part b) and the maximum principle that, 
\begin{equation*}
H(x,t;y,T)\geq \frac{e^{-\frac{L_{w}(x,\tau)}{4\tau}}}{(4\pi\tau)^{n/2}}=\frac{e^{-\frac{L_{w}(x,T-t)}{4\tau}}}{(4\pi(T-t))^{n/2}}.
\end{equation*}
Hence we have,
\begin{equation*}
h(x,t;y,T)\leq \frac{L_{w}(x,\tau)}{4\tau}=\ell_{w}(x,\tau)=\ell_{w}(x,T-l).
\end{equation*}
\end{proof}
A direct consequence is the following estimate on the heat kernel.
\begin{lemma}
\label{integralboundabove} We have $\int_{N}hH\Phi d\mu_{N}\leq \frac{n}{2}\Phi(y,T)$, i.e, $\int_{N}(h-\frac{n}{2})H\Phi d\mu_{N}\leq 0.$
\end{lemma}

\begin{proof}
By lemma \ref{compareRDandconjheat} we have
\begin{align*}
\limsup_{\tau\rightarrow 0}\int_{N}hH\Phi d\mu_{N}\leq \limsup_{\tau\rightarrow 0}\int_{N} \ell_{w}(x,\tau)H\Phi d\mu_{N}(x)\\
\leq \limsup_{\tau\rightarrow 0}\int_{N} \frac{d_{T}^2(x,y)}{4\tau}H\Phi d\mu_{N}(x).
\end{align*}
Using Lemma \ref{asymptoticconj}, 
\begin{equation*}
\lim_{\tau\rightarrow 0}\int_{N} \frac{d_{T}^2(x,y)}{4\tau}H\Phi d\mu_{N}(x)=\lim_{\tau\rightarrow 0}\int_{N} \frac{d_{T}^2(x,y)}{4\tau}\frac{e^{-\frac{d_{T}^2(x,y)}{4\tau}}}{(4\pi\tau)^{n/2}}\Phi d\mu_{N}(x).
\end{equation*}
Either by differentiating twice under the integral sign or using these following identities on Euclidean spaces
\[\int_{-\infty}^{\infty}e^{-a\textbf{x}^2}d\textbf{x} =\sqrt{\frac{\pi}{a}} \text{ and } \int_{-\infty}^{\infty}\textbf{x}^{2}e^{-a\textbf{x}^2}d\textbf{x} =\frac{1}{2a}\sqrt{\frac{\pi}{a}},
\]
we obtain
\begin{equation*}
\int_{\RR^{n}}|x|^{2}e^{-a|x|^2}dx=n(\int_{-\infty}^{\infty}\textbf{x}^{2}e^{-a\textbf{x}^2}d\textbf{x})\Big(\int_{-\infty}^{\infty}e^{-a\textbf{x}^2}d\textbf{x}\Big)^{n-1}=\frac{n}{2a}(\frac{\pi}{a})^{n/2}.
\end{equation*}
Therefore,
\begin{equation*}
\lim_{\tau\rightarrow 0}\frac{d_{T}^2(x,y)}{4\tau}\frac{e^{-\frac{d_{T}^2(x,y)}{4\tau}}}{(4\pi\tau)^{n/2}}=\frac{n}{2}\delta_{y}(x)
\end{equation*}
and so 
\begin{equation*}
\lim_{\tau\rightarrow 0}\int_{N} \frac{d_{T}^2(x,y)}{4\tau}\frac{e^{-\frac{d_{T}^2(x,y)}{4\tau}}}{(4\pi\tau)^{n/2}}\Phi d\mu_{N}(x)=\frac{n}{2}\Phi(y,T).
\end{equation*}
Thus the result follows.
\end{proof}
\begin{remark} In fact, the equality actually holds (See the proof of Theorem \ref{conjandRD}).
\end{remark}
\begin{proposition} 
\label{limit0}
Let $v=\Big((T-t)(2\triangle{h}-|\nabla{h}|^2+S)+h-n\Big)H$ then  \\
{\bf a.} $\Box_{w}^{\ast}v=-2(T-t)\Big(|\mathcal{S}+\text{Hess}h-\frac{g}{2\tau}|^2+|\triangle u-\nabla{u}\nabla{h}|^2\Big)H\leq 0;$\\
{\bf b.} If  $\rho_{\Phi}(t)=\int_{N}v\Phi d\mu_{N}$, then $\lim_{t\rightarrow T}\rho_{\Phi}(t)=0$.
\end{proposition}

\begin{proof}
{\bf a.} 
Let $q=2\triangle{h}-|\nabla{h}|^2+S$ then 
\begin{align*}
H^{-1}\Box_{w}^{\ast}v&=-(\partial_{t}+\Delta)(\tau q+h)-2\left\langle{\nabla(\tau q+h),H^{-1}\nabla{H}}\right\rangle\\
&=q-\tau(\partial_{t}+\Delta)q-(\partial_{t}+\Delta)h+2\tau\left\langle{\nabla q,\nabla{h}}\right\rangle+2|\nabla{h}|^2.
\end{align*}
As H satisfies $\Box_{w}^{\ast}H=0$, $(\partial_{t}+\Delta)h=-S+|\nabla{h}|^2+\frac{n}{2\tau}$.
We compute
\begin{align*}
(\partial_{t}+\Delta)\Delta{h}&=\Delta\frac{\partial{h}}{\partial{t}}+2\left\langle{\mathcal{S},\text{Hess}(h)}\right\rangle-2\Delta{u}\left\langle{\nabla{u},\nabla{h}}\right\rangle+\Delta(\Delta h)\\
&=\Delta(-\Delta{h}+|\nabla{h}|^2-S+\frac{n}{2\tau}+\Delta(\Delta h)\\
&+2\left\langle{\mathcal{S},\text{Hess}(h)}\right\rangle-2\Delta{u}\left\langle{\nabla{u},\nabla{h}}\right\rangle\\
&=\Delta(|\nabla{h}|^2-S)+2\left\langle{\mathcal{S},\text{Hess}(h)}\right\rangle-2\Delta{u}\left\langle{\nabla{u},\nabla{h}}\right\rangle,
\end{align*}
where we use Lemma \ref{evolLap}.
\begin{align*}
(\partial_{t}+\Delta)|\nabla{h}|^2=&2\mathcal{S}(\nabla{h},\nabla{h})+2\left\langle{\nabla{h},\nabla{\frac{\partial{h}}{\partial{t}}}}\right\rangle+\Delta{|\nabla{h}|^2}\\
=&2\left\langle{\nabla{h},\nabla(-\triangle{h}+|\nabla{h}|^2-S)}\right\rangle\\
&+2\mathcal{S}(\nabla{h},\nabla{h})+\Delta{|\nabla{h}|^2}.
\end{align*}
Recall from (\ref{evolS}), $(\partial_{t}+\Delta)S=2\Delta{S}+2|\mathcal{S}|^2+2|\Delta{u}|^2,$ and
\begin{align*}
2\mathcal{S}(\nabla{h},\nabla{h})&=2\text{Rc}(\nabla{h},\nabla{h})-2du\otimes du(\nabla{h},\nabla{h})=2\text{Rc}(\nabla{h},\nabla{h})-2\left\langle{\nabla{u},\nabla{h}}\right\rangle^2\\
\Delta{|\nabla{h}|^2}&= 2\text{Hess}(h)^2+2\left\langle{\nabla{h},\nabla{\Delta{h}}}\right\rangle+2\text{Rc}(\nabla{h},\nabla{h}),
\end{align*}
where the second equation is by Bochner's identity. Combining those above yields
\begin{align*}
(\partial_{t}+\Delta)q=&4\left\langle{\mathcal{S},\text{Hess}(h)}\right\rangle-4\Delta{u}\left\langle{\nabla{u},\nabla{h}}\right\rangle+\Delta|\nabla{h}|^2\\
&-2\mathcal{S}(\nabla{h},\nabla{h})-2\left\langle{\nabla{h},\nabla(-\triangle{h}+|\nabla{h}|^2-S)}\right\rangle+2|\mathcal{S}|^2+2|\Delta{u}|^2\\
=&4\left\langle{\mathcal{S},\text{Hess}(h)}\right\rangle-4\Delta{u}\left\langle{\nabla{u},\nabla{h}}\right\rangle+2\left\langle{\nabla{h},\nabla q}\right\rangle+2\text{Hess}(h)^2\\
&+2|\mathcal{S}|^2+2|\Delta{u}|^2+2\left\langle{\nabla{u},\nabla{h}}\right\rangle^2\\
=&2|\mathcal{S}+\text{Hess}(h)|^2+2|\Delta{u}-\left\langle{\nabla{u},\nabla{h}}\right\rangle|^2+2\left\langle{\nabla{h},\nabla q}\right\rangle.
\end{align*}
Thus, 
\begin{align*}
H^{-1}\Box_{w}^{\ast}v&=q+S-|\nabla{h}|^2-\frac{n}{2\tau}+2|\nabla{h}|^2\\
&-2\tau(|\mathcal{S}+\text{Hess}(h)|^2+2|\Delta{u}-\left\langle{\nabla{u},\nabla{h}}\right\rangle|^2)\\
&=-2\tau\Big(|\mathcal{S}+\text{Hess}(h)-\frac{g}{2\tau}|^2+|\triangle u-\nabla{u}\nabla{h}|^2\Big).
\end{align*}
The result follows.\\

{\bf b.} IBP yields
\begin{align*}
\rho_{\Phi}(t) &=\int_{N}\Big(\tau(2\triangle{h}-|\nabla h|^2+S)+h-n\Big)H\Phi d\mu_{N}\\
&=-\int_{N}2\tau\nabla{h}\nabla({H\Phi})d\mu_{N}-\int_{N}\tau|\nabla h|^2 H\Phi d\mu_{N}+\int_{N}(\tau S+h-n)H\Phi d\mu_{N}\\
&= \int_{N}\tau|\nabla h|^2 H\Phi d\mu_{N}-2\tau\int_{N} \nabla{\Phi}\nabla{h}H d\mu_{N}+\int_{N}(\tau S+h-n)H\Phi d\mu_{N}\\
&= \int_{N}\tau|\nabla h|^2 H\Phi d\mu_{N}-2\tau\int_{N}H\triangle{\Phi}d\mu_{N}+\int_{N}(\tau S+h-n)H\Phi d\mu_{N}\\
&= \int_{N}\tau|\nabla h|^2 H\Phi d\mu_{N}+\int_{N}hH\Phi d\mu_{N}-2\tau\int_{N}H\triangle{\Phi}d\mu_{N}+\int_{N}(\tau S-n)H\Phi d\mu_{N}.
\end{align*}
For the first term, using Lemmas \ref{conjestimateS} and \ref{gradconjS} for $N\times[\frac{\tau}{2},\tau]$ to arrive at
\begin{align*}
\tau\int_{N}|\nabla h|^2H\Phi d\mu_{N}&\leq (2+C_{1}\tau)\int_{N}(\ln{(\frac{C_{3}e^{-D\tau/3}}{H(4\pi\tau)^{n/2}})}+C_{2}\tau)H\Phi d\mu_{N}\\
& \leq (2+C_{1}\tau)\int_{N}(\ln{C_{3}}-\frac{D\tau}{3}+h+C_{2}\tau)H\Phi d\mu_{N},
\end{align*}
with $C_{1},C_{2}$ as in Lemma \ref{gradconjS} while $C_{3}=\frac{e^{B}}{2^{n/2}}$.\\
Therefore, applying Lemma \ref{integralboundabove}, 
\begin{align*}
\lim_{\tau\rightarrow 0}(\int_{N}\tau|\nabla h|^2d\mu_{N}+\int_{N}hH\Phi d\mu_{N})&\leq 3\int_{N}hH\Phi d\mu_{N}+2\ln{C_{3}}\Phi(x,T)\\
&\leq (\frac{3n}{2}+2\ln{C_{3}})\Phi(x,T). 
\end{align*}
Now we observe that expect for the first 2 terms, the rest approaches $-n\Phi(y,T)$ as $\tau\rightarrow 0$. Thus 
\begin{equation*}
\lim_{t\rightarrow T}\rho_{\Phi}(t)\leq C_{4}\Phi(x,T).
\end{equation*}
Furthermore, since $\Phi$ is a positive test function satisfying the heat equation $\partial_{t}\Phi=\triangle{\Phi}$, hence,
\begin{equation}
\label{evolintegral}
\partial_{t}\rho_{\Phi}(t)=\partial_{t}\int_{N}v\Phi d\mu_{N}=\int_{N}(\Box{\Phi}v-\Phi\Box_{w}^{\ast}v)d\mu_{N}\geq 0.
\end{equation}
The above conditions imply that there exists $\alpha$ such that 
\begin{equation*}
\lim_{t\rightarrow T}\rho_{\Phi}(t)=\alpha.
\end{equation*}
Hence $\lim_{\tau\rightarrow 0}(\rho_{\Phi}(T-\tau)-\rho_{\Phi}(T-\frac{\tau}{2}))=0$. By equation (\ref{evolintegral}), part a), and the mean-value theorem, there exists a sequence $\tau_{i}\rightarrow 0$ such that 
\begin{equation*}
\lim_{\tau_{i}\rightarrow 0}\tau_{i}^2\int_{N}\Big(|\mathcal{S}+\text{Hess}h-\frac{g}{2\tau}|^2+|\triangle u-\nabla{u}\nabla{h}|^2\Big)H\Phi d\mu_{N}=0.
\end{equation*}
Now using standard inequalitites yield,
\begin{align*}
&(\int_{N}\tau_{i}(S+\triangle{h}-\frac{n}{2\tau_{i}})H\Phi d\mu_{N})^2 \\
&\leq (\int_{N}\tau_{i}^2(S+\triangle{h}-\frac{n}{2\tau_{i}})^2H\Phi d\mu_{N})(\int_{N}H\Phi d\mu_{N})\\
& \leq  (\int_{N}\tau_{i}^2|\mathcal{S}+\text{Hess}h-\frac{g}{2\tau}|^2H\Phi d\mu_{N})(\int_{N}H\Phi d\mu_{N}).
\end{align*} 
Since $\lim_{\tau_{i}\rightarrow 0}\int_{N}H\Phi d\mu_{N}=\Phi(y,T)<\infty$ and $|\triangle u-\nabla{u}\nabla{h}|^2\geq 0$, 
\begin{equation*}
\lim_{\tau_{i}\rightarrow 0}\int_{N}\tau_{i}(S+\triangle{h}-\frac{n}{2\tau_{i}})H\Phi d\mu_{N}=0.
\end{equation*}
Therefore, by Lemma \ref {integralboundabove},
\begin{align*}
\lim_{t\rightarrow T}\rho_{\Phi}(t)&=\lim_{t\rightarrow T}\int_{N}(\tau_{i}(2\triangle{h}-|\nabla{h}|^2+S)+h-n)H\Phi d\mu_{N}\\
&=\lim_{t\rightarrow T}\int_{N}(\tau_{i}(\triangle{h}-|\nabla{h}|^2)+h-\frac{n}{2})H\Phi d\mu_{N}\\
&=\lim_{t\rightarrow T}(\int_{N}(-\tau_{i}H\triangle{\Phi}d\mu_{N}+\int_{N}(h-\frac{n}{2})H\Phi d\mu_{N})\\
&=\int_{N}(h-\frac{n}{2})H\Phi d\mu_{N}\leq 0. 
\end{align*}
So $\alpha\leq 0$. To show that equality holds, we proceed by contradiction. Without loss of generality, we may assume $\Phi(y,T)=1$. Let $H\Phi=(4\pi\tau)^{-n/2}e^{\tilde{h}}$ (that is, $\tilde{h}=h-\ln{\Phi}$), then IBP yields,
\begin{equation}
\rho_{\Phi}(t)=\Psi_{w}(g,u,\tau,\tilde{h})+\int_{N}\Big(\tau(\frac{|\nabla{\Phi}|^2}{\Phi})-\Phi\ln{\Phi}\Big)H d\mu_{N}.
\end{equation}
By the choice of $\Phi$ the last term converges to 0 as $\tau\rightarrow 0$. So if $\lim_{t\rightarrow T}\rho_{\Phi}(t)=\alpha<0$ then $\lim_{\tau\rightarrow 0}\mu_{w}(g,u,\tau)<0$ and, thus, contradictss Lemma \ref{behaviorforsmalltau}. Therefore $\alpha=0$.
\end{proof}
Now Theorem \ref{eqstate} follows immediately.
\begin{proof}(Theorem \ref{eqstate})
 Recall from inequality (\ref{evolintegral})  
$$\partial_{t}\int_{N}v\Phi d\mu_{N}=\int_{N}(\Box{h}v-h\Box_{w}^{\ast}v)d\mu_{N}\geq 0.$$
By Proposition \ref{limit0}, $\lim_{t\rightarrow T}\int_{N}v\Phi d\mu_{N}=0$. Since $\Phi$ is arbitrary, $v\leq 0$.
\end{proof}

\section{ Proofs of Main Results}
\begin{proof}(Theorem \ref{harnackwarpedlaplacian})
By the diffeomorphism discussed in Section 2, the result follow from Theorem \ref{eqstate}. Note that if, with respect to (\ref{warpedLie}), $\Phi$ ($H$) is a positive function satisfying the equation $\partial_{t}\Phi=\triangle_{N}{\Phi}$ ($\Box^{\ast}_{w}H=0$) then pulling back by the diffeomorphism, with respect to (\ref{warpedbeforelie}),
\begin{align*}
\partial_{t}\Phi&=\triangle_{N}{\Phi}+p\nabla{u}\nabla{\Phi}=\triangle_{g_{M}}\Phi,\\
\partial_{t}H&=-\triangle_{N}H+SH+p\nabla{u}\nabla{H}.
\end{align*}

\end{proof}

\begin{proof}(Corollary \ref{harnackwarpedtime})
As H satisfies $\Box_{w}^{\ast}H=0$,  
\begin{equation*} h_{t}=-S-\triangle{h}+|\nabla{h}|^2+\frac{n}{2\tau}.
\end{equation*}
Substituting that into $\partial_{t}h(\gamma(t),t)=\nabla{h}\dot{\gamma}(t)+h_{t}\geq h_{t}-\frac{1}{2}(|\nabla{h}|^2+|\dot{\gamma}(t)|^2)$ and applying $v\leq 0$ prove the result. 
\end{proof}

\begin{proof}(Theorem \ref{conjandRD})
Part a) and b) are proved in Lemma \ref{compareRDandconjheat}. Part c) follows from Lemma \ref{integralboundabove} and the proof of Lemma \ref{limit0}, where it is shown that equality must hold.
\end{proof}

\begin{proof}(Corollary \ref{strangebehavior})
We abuse notation here by writing, 
$$\Psi(g_{M},\tau,\overline{h})=\int_{M}\Big(\tau(|\nabla \overline{h}|^2+\RR_{M})+\overline{h}-n-p\Big)(4\pi\tau)^{-(n+1)/2}e^{-\overline{h}}d\mu_{M},$$
for $\int_{M}\overline{H}d\mu_{M}=V(F)$.\\
Let $\Phi=1$ be the constant function in proposition \ref{limit0} then $\rho_{1}(t)=\Psi_{w}(g,u,\tau,h)$.\\
By Lemmas \ref{baseandwhole} and \ref{monoformwarp},

\begin{align*}
\Psi(g_{M},\tau,\overline{h})&= V(F)\int_{N}\Big(\tau(S+|\nabla{h}|^2)+h-n-p+pu-\frac{p}{2}\ln(4\pi\tau))H d\mu_{N}\\
&=V(F)\Psi_{w}(g,u,\tau,h)+pV(F)\int_{N}(u-1-\frac{1}{2}\ln(4\pi\tau))H d\mu_{N}.
\end{align*}

Since $\lim_{\tau\rightarrow 0}\ln(4\pi\tau)=-\infty$, by Lemma \ref{limit0}, $\lim_{\tau\rightarrow 0}\Psi(g_{M},\tau,\overline{h})=+\infty$. A direct calculation yields that,
\begin{equation}
\Psi(g_{M},\tau,\tilde{h})=\frac{1}{V(F)}\Psi(g_{M},\tau,\overline{h})+\ln(V(F)).
\end{equation}
Thus the result follows.
\end{proof}

\def\cprime{$'$}
\bibliographystyle{plain}
\bibliography{bio}

\end{document}